\documentclass[reqno]{amsart}
\usepackage{hyperref}
\newtheorem{Theorem}{Theorem}[section]

\newtheorem{Lemma}[Theorem]{Lemma}

\theoremstyle{remark}

\numberwithin{equation}{section}

\begin{document}

\title[Classical solution of the Cauchy problem for biwave equation]{CLASSICAL SOLUTION OF THE CAUCHY PROBLEM FOR BIWAVE EQUATION: APPLICATION OF FOURIER TRANSFORM}
\author{Victor Korzyuk}
\address{{Belarusian State University}, 4, Fr. Scorina Ave., 220030 Minsk, Belarus}
  \email{korzyuk@bsu.by}
\author{Nguyen Van Vinh}
\address{{Hue University of Education}, 34, Le Loi, 530000 Hue, Vietnam}
  \email{vinhnguyen0109@gmail.com}%
 \author{Nguyen Tuan Minh}
   \address{{Belarusian State University}, 4, Fr. Scorina Ave., 220030 Minsk, Belarus}
\email{minhnguyen@yandex.ru}

\keywords{Biwave equation, Fourier transform, Cauchy problem}
\subjclass[2000]{35G05, 35G10}

 \begin{abstract}
In this paper, we use some Fourier analysis techniques to find an exact solution to the Cauchy problem for the $n$-dimensional biwave equation in the upper half-space $\mathbb{R}^n\times [0,+\infty)$.
 \end{abstract}

 \maketitle

\section{Introduction}\label{s:1}

The Cauchy initial value problem for the $n$-dimensional biwave equation consists in finding a scalar function $u \in C^4(\mathbb{R}^n\times[0,+\infty))$  such that for $(x,t) \in \mathbb{R}^n\times (0,+\infty)$ then
\begin{equation}\label{equ1}
\left( {\frac{\partial^2 }{{\partial t^2 }} - a^2 \rm{\Delta} } \right)\left( {\frac{{\partial^2}}{{\partial t^2 }} - b^2 \rm{\Delta}} \right) u(x,t)= f(x,t),\;a^2 > b^2 > 0,
\end{equation}
together with the initial conditions
\begin{equation}\label{equ2}
 u(x,0) = \phi _0(x), \frac{\partial u}{\partial t}(x,0)  = \phi _1(x), \frac{\partial ^2 u}{\partial t^2 } (x,0)  = \phi _2(x),  \frac{\partial ^3 u}{\partial t^3 } (x,0)= \phi _3(x),
\end{equation}
for $(x,t)\in \mathbb{R}^n\times\{0\}$.\\

 The biwave equation has been studied in some models related to the mathematical theory of elasticity. Let us consider the mathematical formulation for the displacement equation of a homogeneous isotropic elastic body. Remark that, the Newton's second law leads to the Cauchy's motion equation of an elastic body, which takes the form
\begin{equation}\label{iso1}\boldsymbol{\nabla}\cdot\boldsymbol{\sigma} + \mathbf{f} = \rho~\ddot{\mathbf{u}},\end{equation}
where $ \boldsymbol{\sigma}$ is the Cauchy stress tensor field, $\mathbf{u}$ is the displacement vector field, $\mathbf{f}$ is the vector field of body force per unit volume and $\rho$ is the mass density.\\
The infinitesimal strain tensor field is given by the equation
\begin{equation}\label{iso2}\boldsymbol{\varepsilon}  =\tfrac{1}{2}
\left[\boldsymbol{\nabla}\mathbf{u}+(\boldsymbol{\nabla}\mathbf{u})^T\right].\end{equation}
Moreover, the Hooke's law for homogeneous isotropic bodies has the form
 \begin{equation}\label{iso3}  \boldsymbol{\sigma} = \lambda~\mathrm{trace}(\boldsymbol{\varepsilon})~\mathbf{I} + 2\mu~\boldsymbol{\varepsilon} ,\end{equation}
where $\lambda, \mu>0$ are Lam\'e's parameters and $\mathbf{I}$ is the second-order identity tensor.
Substituting the strain-displacement equation (\ref{iso2}) and the Hooke's equation (\ref{iso3}) into the equilibrium equation (\ref{iso1}), we obtain the Navier's elastodynamic wave equation
\begin{equation}\label{bi2}(\lambda+\mu)\nabla\text{div}(\mathbf{u})+\mu{\rm\Delta}\mathbf{u}+\mathbf{f}= \rho~\ddot{\mathbf{u}}.\end{equation}
This equation in the Cartesian coordinates has the form
$$(\lambda+\mu)\frac{\partial }{\partial x_k}\left(\sum_{j=1}^n\frac{\partial u_j}{\partial x_j}\right)+\mu{\rm \Delta} u_k + f_k = \rho \frac{\partial^2 u_k}{\partial t^2}, k=\overline{1,n}.$$
Let us denote $a^2=(\lambda+2\mu)/\rho, b^2=\mu/\rho$, then (\ref{bi2}) can be rewritten as
\begin{equation}\mathcal{L} \equiv \label{bi3} \left( \frac{\partial^2 }{\partial t^2} -b^2{\rm\Delta} \right)\mathbf{u}-(a^2-b^2)\nabla\text{div}(\mathbf{u})-\mathbf{f}/\rho =0.\end{equation}
It is easy to show that the equation (\ref{bi3}) has a solution in the following form
\begin{equation}\label{bi4} \mathbf{u}=\left( \frac{\partial^2 }{\partial t^2} -a^2{\rm\Delta} \right)\mathbf{w}+(a^2-b^2)\nabla\text{div}(\mathbf{w}),\end{equation}
where $\mathbf{w}$ is a solution to the biwave equation
\begin{equation}\label{bi5} \left( \frac{\partial^2 }{\partial t^2} -a^2{\rm\Delta} \right)\left( \frac{\partial^2 }{\partial t^2} -b^2{\rm\Delta} \right)\mathbf{w}=\mathbf{f}/\rho.\end{equation}
This formula is called as Cauchy-Kovalevski-Somigliana solution to the elastodynamic wave equation. Indeed, substituting (\ref{bi4})-(\ref{bi5}) to the left-hand side of (\ref{bi3}), we have
$$\mathcal{L}=\left( \frac{\partial^2 }{\partial t^2} -b^2 {\rm \Delta} \right)\left(\left( \frac{\partial^2 }{\partial t^2} -a^2 {\rm\Delta}  \right)\mathbf{w}+(a^2-b^2)\nabla \text{div}(\mathbf{w}) \right)$$
$$- (a^2-b^2)\nabla\left( \left( \frac{\partial^2 }{\partial t^2} -a^2{\rm \Delta}  \right)\text{div}(\mathbf{w})+(a^2-b^2){\rm\Delta} \text{div} (\mathbf{w}) \right)-\mathbf{f}/\rho.$$
Note that
$$\left( \frac{\partial^2 }{\partial t^2} -a^2{\rm\Delta} \right)\text{div}(\mathbf{w})+(a^2-b^2){\rm \Delta} \text{div}(\mathbf {w}) = \left( \frac{\partial^2 }{\partial t^2} -b^2{\rm\Delta}\right)\text{div}(\mathbf{w}).
$$
Therefore, we get
$$\mathcal{L}=\left( \frac{\partial^2 }{\partial t^2} -b^2{\rm \Delta} \right)\left( \frac{\partial^2 }{\partial t^2} -a^2{\rm\Delta} \right)\mathbf{w}+(a^2-b^2) \left( \frac{\partial^2 }{\partial t^2} -b^2{\rm\Delta} \right)\nabla\text{div}(\mathbf{w})$$
$$-(a^2-b^2) \nabla\left( \frac{\partial^2 }{\partial t^2} -b^2{\rm \Delta} \right)\text{div}(\mathbf{w}) - \mathbf{f}/\rho=0.$$

 For more explanations about physical context, we refer the reader to \cite{Hetnarski2004, Musk2010, Sommerfeld1950}. Actually, there are not many mathematical papers related to biwave equations because it gets more difficult when studying high-order PDEs. In some recent researches, the symmetry analysis of biwave equations is considered and exact solutions are obtained by Fushchych, Roman and Zhdanov \cite{Fushchych1996}; the existence and uniqueness of the solution to Cauchy initial value problem, bounded valued problem are given by Korzyuk, Cheb and Konopelko \cite{ Korzyuk2010,Korzyuk2007}; the finite element methods for approximations of biwave equation are developed by Feng and Neilan \cite{Feng2010, Feng2011}. In our present work, the main result is to show the exact classical solution to the Cauchy initial value problem for the $n$-dimensional biwave equation by using some techniques of Fourier analysis.\\

Returning to the Cauchy problem for the biwave equation (\ref{equ1}), we suppose that $\phi_0,\phi_1,\phi_2,\phi_3,$ and $f$ are elements in Schwartz space $\mathcal{S}(\mathbb{R}^n)$ of rapidly decreasing functions on $\mathbb{R}^n$. Remark that, an indefinitely differentiable function $\phi$ is called rapidly decreasing when $\phi$ and all its derivatives are required to satisfy that
\[||\phi||_{\alpha,\beta}=\sup_{x\in\mathbb{R}^n}\left| x^{\alpha}\left(\frac{\partial}{\partial x}\right)^{\beta} \phi(x) \right|<\infty,\]
for every multi-index $\alpha$ and $\beta$. The Fourier transform of Schwartz function $\phi\in \mathcal{S}(\mathbb{R}^n)$ is defined by
\[
\mathcal{F}\left[ {\phi} \right](\xi) \equiv \widehat{\phi}\left( \xi  \right) = \int\limits_{\mathbb{R}^n} {e^{ - ix.\xi } \phi(x)dx}.
\]
The convolution of two integrable functions $\phi$ and $\psi$ is written as $\phi*\psi$.  It is defined as the integral of the product of the two functions after one is reversed and shifted. As such, it is a particular kind of integral transform:
\[
\left( {\phi*\psi} \right)\left( t \right) = \int\limits_{\mathbb{R}^n} {\phi\left( \tau  \right)} \psi\left( {t - \tau } \right)d\tau.
\]
In the Euclidean space $\mathbb{R}^n$, the spherical mean of an integrable function $\phi$ around a point $x$ is the average of all values of that function on a sphere of radius $R$ centered at that point, i.e. it is defined by the formula
\[\mathcal{M}_R(\phi)(x)=\frac{1}{\omega_{n}R^{n-1}}\int_{\partial B(x,R)}\phi(y)d\sigma(y)\equiv \frac{1}{\omega_{n}}\int_{\partial B(0,1)}\phi(x+Ry)d\sigma(y),\]
where $\omega_{n}$ is the surface area of the $n$-dimensional unit ball and $\sigma$ is the spherical measure area.

\section{Main results}\label{s:2}

The Cauchy problem for the homogeneous biwave equation in $\mathbb{R}^n\times [0,+\infty)$ that we will be studying in this section, reads as follows
\begin{equation}\label{equ3}
\left( {\frac{\partial^2 }{{\partial t^2 }} - a^2{\rm\Delta} } \right)\left( {\frac{{\partial^2 u}}{{\partial t^2 }} - b^2{\rm \Delta} u} \right) =0, \ a^2>b^2>0,
\end{equation}
with the initial conditions
\begin{equation}\label{equ4}
\left. u \right|_{t = 0}  = \phi _0 \left( x \right),\left. {\frac{{\partial u}}{{\partial t}}} \right|_{t = 0}  = \phi _1 \left( x \right),\left. {\frac{{\partial ^2 u}}{{\partial t^2 }}} \right|_{t = 0}  = \phi _2 \left( x \right),\left. {\frac{{\partial ^3 u}}{{\partial t^3 }}} \right|_{t = 0}  = \phi _3 \left( x \right),
\end{equation}
where $\phi_0,\phi_1,\phi_2,\phi_3$ are Schwartz functions.\\\\
The equation (\ref{equ3}) can be rewritten as a fourth-order PDE, which has the following form
\begin{equation}\label{equ5}
\frac{\partial^4 u}{\partial t^4 }-(a^2+b^2)\frac{\partial^2}{\partial t^2 }{\rm\Delta} u+a^2b^2{\rm\Delta}^2 u=0.
\end{equation}
Taking Fourier transform to the both sides of the equation (\ref{equ5}), we obtain
\[
\frac{\partial^4 }{\partial t^4 }\widehat{u}(\xi,t)+(a^2+b^2)|\xi|^2 \frac{\partial^2}{\partial t^2 }\widehat{u}(\xi,t) +a^2b^2 |\xi|^4 \widehat{u}(\xi,t)=0.
\]
This fourth ODE has the general solution, which takes the form
\[
\widehat{u}\left( {\xi ,t} \right) = C_1 \cos \left( {a|\xi| t} \right) + C_2 \sin \left( {a|\xi| t} \right) + C_3 \cos \left( {b|\xi| t} \right) + C_4 \sin \left( {b|\xi| t} \right),
\]
where parameters $C_1,C_2,C_3, C_4$ are determined from the initial conditions:
\[
\left. {\widehat{u}\left( {\xi ,t} \right)} \right|_{t = 0}  = C_1  + C_3  = \widehat{\phi_0} \left( \xi  \right),\ \left. {\frac{{\partial \widehat{u}\left( {\xi ,t} \right)}}{{\partial t}}} \right|_{t = 0}  = aC_2 |\xi|  + bC_4 |\xi|  = \widehat{\phi_1} \left( \xi  \right),
\]
\[
\left. {\frac{{\partial ^2 \widehat{u}\left( {\xi ,0} \right)}}{{\partial t^2 }}} \right|_{t = 0}  =  - a^2 C_1 |\xi| ^2  - b^2 C_3 |\xi| ^2  =\widehat{\phi_2} \left( \xi  \right),\]
\[
\left. {\frac{{\partial ^3 \widehat{u}\left( {\xi ,0} \right)}}{{\partial t^3 }}} \right|_{t = 0}  =  - a^3 C_2 |\xi| ^3  - b^3 C_4 |\xi| ^3  = \widehat{\phi_3} \left( \xi  \right).
\]
Solving above system of equations, we easily get the image of solution $u(x,t)$ via Fourier transform given by
\[\begin{array}{c} \displaystyle
\widehat{u}\left( {\xi ,t} \right) =  - \frac{{b^2 |\xi| ^2 \widehat{\phi_0} \left( \xi  \right) + \widehat{\phi_2} \left( \xi  \right)}}{{\left( {a^2  - b^2 } \right)|\xi| ^2 }}\cos \left( {a|\xi| t} \right) - \frac{{b^2 |\xi| ^2 \widehat{\phi_1} \left( \xi  \right) + \widehat{\phi_3} \left( \xi  \right)}}{{\left( {a^3  - ab^2 } \right)|\xi| ^3 }}\sin \left( {a|\xi| t} \right) +
\\ \displaystyle
 + \frac{{a^2 |\xi| ^2 \widehat{\phi_0} \left( \xi  \right) + \widehat{\phi_2} \left( \xi  \right)}}{{\left( {a^2  - b^2 } \right)|\xi|^2 }}\cos \left( {b|\xi| t} \right)\, + \frac{{a^2 |\xi| ^2 \widehat{\phi_1} \left( \xi  \right) +\widehat{\phi_3} \left( \xi  \right)}}{{\left( {a^2 b - b^3 } \right)|\xi| ^3 }}\sin \left( {b|\xi| t} \right),
\end{array}\]
or by the rewritten form
\begin{equation}\label{fou}
\begin{array}{c} \displaystyle \widehat{u}(\xi,t)=-\frac{b^2}{a^2-b^2} \widehat{\phi _0} \left( \xi \right)\cos\left( a|\xi| t \right)+\frac{a^2}{a^2-b^2} \widehat{\phi_0}\left( \xi \right)\cos\left( b |\xi| t \right)-\\ \
\displaystyle-\frac{b^2}{a(a^2-b^2)}\widehat{\phi_1}\left( \xi \right)\frac{\sin\left( a |\xi| t \right)}{|\xi|}+\frac{a^2}{b(a^2-b^2)}\widehat{\phi_1}\left( \xi \right)\frac{\sin\left( b |\xi| t \right)}{|\xi|}+\\
\displaystyle+\frac{\widehat{\phi_3}\left( \xi \right)}{a^2-b^2}\left[ \frac{1}{b}\frac{\sin\left( b |\xi| t \right)}{|\xi|^3} - \frac{1}{a}\frac{\sin\left( a |\xi| t \right)}{|\xi|^3}\right] +\frac{\widehat{\phi_2}\left( \xi \right)}{a^2-b^2}\left[ \frac{\cos\left( b |\xi| t \right)}{|\xi|^2}- \frac{\cos\left( a |\xi| t \right)}{|\xi|^2} \right].
\end{array}
\end{equation}
In the next sequence, we will find the inverse formula of (\ref{fou}) and obtain an exact solution to the equation (\ref{equ3}).
\begin{Theorem}\label{th1}
The Cauchy problem for the homogeneous biwave equation in $\mathbb{R}\times[0,+\infty)$ has the following solution
\[
u(x,t) =\frac{1}{{2ab\left( {a^2  - b^2 } \right)}}\left[  - b^3 \int\limits_{x - at}^{x + at} {\phi _1 (y)dy}  + a^3 \int\limits_{x - bt}^{x + bt} {\phi _1 (y)dy}  - ab\int\limits_{x - at}^{x - bt} {\int\limits_0^y {\phi _2 (u)dudy}  + } \right.\]
\[ + ab\int\limits_{x + bt}^{x + at} {\int\limits_0^y {\phi _2 (u)dudy} }  + b\int\limits_{x - at}^{x + at} {\int\limits_0^y {\int\limits_0^\tau  {\phi _3 (\omega )d\omega d\tau dy} } }  - a\int\limits_{x - bt}^{x + bt} {\int\limits_0^y {\int\limits_0^\tau  {\phi _3 (\omega )d\omega d\tau dy} } } \]
\begin{equation}\label{eq0}
\left.
 - ab^3 \phi _0 \left( {x + at} \right) - ab^3 \phi _0 \left( {x - at} \right) + a^3 b\phi _0 \left( {x + bt} \right) + a^3 b\phi _0 \left( {x - bt} \right)
{ }^{^{^{^{^{_{}^{_{}^{^{^{ } } } } } } } } }  \right ].
\end{equation}
\end{Theorem}
\begin{proof}
We have some facts that
\[\begin{array}{llll}\displaystyle
\cos \left( {a|\xi| t} \right) = \frac{{e^{ia|\xi| t}  + e^{ - ia|\xi| t} }}{2},\ \sin \left( {a|\xi| t} \right) = \frac{{e^{ia|\xi| t}  - e^{ - ia|\xi| t} }}{{2i}},
\\ \displaystyle
\frac{{\sin \left( {a|\xi| t} \right)}}{|\xi| } = \frac{{e^{ia|\xi| t}  - e^{ - ia|\xi| t} }}{{2i|\xi| }} = \frac{1}{2}\int\limits_{ - at}^{at} {e^{i|\xi| \theta } d\theta },
\\ \displaystyle
\frac{{\cos \left( {a|\xi| t} \right)}}{{|\xi| ^2 }} = \frac{{e^{ia|\xi| t}  + e^{ - ia|\xi| t} }}{{2|\xi| ^2 }} =  - \frac{1}{2}\int\limits_0^{at} {\int\limits_0^y {e^{i|\xi| u} du} dy }  - \frac{1}{2}\int\limits_0^{at} {\int\limits_0^y {e^{ - i|\xi| u} du} dy }  + \frac{1}{{|\xi| ^2 }},
\\ \displaystyle
\frac{{\sin \left( {a|\xi| t} \right)}}{{|\xi| ^3 }} = \frac{{e^{ia|\xi| t}  - e^{ - ia|\xi| t} }}{{2i|\xi| ^3 }} =  - \frac{1}{2}\int\limits_{ - at}^{at} {\int\limits_0^y {\int\limits_0^\tau  {e^{i|\xi| u} dud\tau dy} } }  + \frac{{at}}{{|\xi| ^2 }}.
\end{array}\]
Moreover,
\[
\widehat{\delta}\left( {x - \alpha t} \right) = \int\limits_{ - \infty }^{ + \infty } {e^{ - i|\xi| x} \delta \left( {x - \alpha t} \right)} dx = e^{ - i\alpha |\xi| t},
\]
where $\delta \left( x \right)$ is the Dirac delta function. Hence, by the property of the Dirac's delta function, we note that
\[
\left( {\frac{{e^{ia|\xi| t}  + e^{ - ia|\xi| t} }}{2}} \right)\widehat{\phi _0 }\left( \xi  \right) = \frac{{\mathcal{F}\left[ {\left( {\delta \left( {x + at} \right) * \phi _0 \left( x \right)} \right)} \right] + \mathcal{F}\left[ {\left( {\delta \left( {x - at} \right) * \phi _0 \left( x \right)} \right)} \right]}}{2},
\]
and
\[
\widehat{\phi _1 }\left( \xi  \right)\left( {\frac{1}{2}\int\limits_{ - at}^{at} {e^{i|\xi| \theta } d\theta } } \right) = \frac{1}{2}\int\limits_{-at}^{at} {\mathcal{F}\left[ {\delta \left( {x + \theta } \right) * \phi _1 \left( x \right)} \right]d\theta },
\]
\[
 \widehat{\phi _2 }\left( \xi  \right)\left( { - \frac{1}{2}\int\limits_0^{at} {\int\limits_0^y {e^{i|\xi| u} du} dy}  - \frac{1}{2}\int\limits_0^{at} {\int\limits_0^y {e^{ - i|\xi| u} du} dy} } \right)\]
 \[
  = \left( { - \frac{1}{2}\int\limits_0^{at} {\int\limits_0^y {\mathcal{F}\left[ {\delta \left( {x + u} \right) * \phi _2 \left( x \right)} \right]du} dy}  - \frac{1}{2}\int\limits_0^{at} {\int\limits_0^y {\mathcal{F}\left[ {\delta \left( {x - u} \right) * \phi _2 \left( x \right)} \right]du} dy} } \right) ,
\]
\[
\widehat{\phi _3 }\left( \xi  \right)\left( { - \frac{1}{2}\int\limits_{ - at}^{at} {\int\limits_0^y {\int\limits_0^\tau  {e^{i|\xi| u} dud\tau dy} } } } \right) =  - \frac{1}{2}\int\limits_{ - at}^{at} {\int\limits_0^y {\int\limits_0^\tau  {\mathcal{F}\left[ {\delta \left( {x + u} \right) * \phi _3 \left( x \right)} \right]dud\tau dy} } }.
\]
Substituting the above identities into the formula (\ref{fou}), we obtain that
\[
\widehat{u}\left( {\xi ,t} \right) =  - \frac{{b^2 }}{{\left( {a^2  - b^2 } \right)}}\frac{{\mathcal{F}\left[ {\delta \left( {x + at} \right) * \phi _0 \left( x \right)} \right] + \mathcal{F}\left[ {\delta \left( {x - at} \right) * \phi _0 \left( x \right)} \right]}}{2}
\]
\[
 + \frac{{a^2 }}{{\left( {a^2  - b^2 } \right)}}\frac{{\mathcal{F}\left[ {\delta \left( {x + bt} \right) * \phi _0 \left( x \right)} \right] + \mathcal{F}\left[ {\delta \left( {x - bt} \right) * \phi _0 \left( x \right)} \right]}}{2}
\]
\[
 - \,\frac{{b^2 }}{{\left( {a^3  - ab^2 } \right)}}\frac{1}{2}\int\limits_{ - at}^{at} {\mathcal{F}\left[ {\delta \left( {x + \theta } \right) * \phi _1 \left( x \right)} \right]d\theta }  + \frac{{a^2 }}{{\left( {a^2 b - b^3 } \right)}}\frac{1}{2}\int\limits_{ - bt}^{bt} {\mathcal{F}\left[ {\delta \left( {x + \theta } \right) * \phi _1 \left( x \right)} \right]d\theta }
\]
\[
 - \frac{1}{{\left( {a^2  - b^2 } \right)}}\left( { - \frac{1}{2}\int\limits_0^{at} {\int\limits_0^y {\mathcal{F}\left[ {\delta \left( {x + u} \right) * \phi _2 \left( x \right)} \right]du} dy}  - \frac{1}{2}\int\limits_0^{at} {\int\limits_0^y {\mathcal{F}\left[ {\delta \left( {x - u} \right) * \phi _2 \left( x \right)} \right]du} dy} } \right)
\]
\[
 + \frac{1}{{\left( {a^2  - b^2 } \right)}}\left( { - \frac{1}{2}\int\limits_0^{bt} {\int\limits_0^y {\mathcal{F}\left[ {\delta \left( {x + u} \right) * \phi _2 \left( x \right)} \right]du} dy}  - \frac{1}{2}\int\limits_0^{bt} {\int\limits_0^y {\mathcal{F}\left[ {\delta \left( {x - u} \right) * \phi _2 \left( x \right)} \right]du} dy} } \right)
\]
\[
 + \,\frac{1}{{\left( {a^3  - ab^2 } \right)}}\frac{1}{2}\int\limits_{ - at}^{at} {\int\limits_0^y {\int\limits_0^\tau  {\mathcal{F}\left[ {\delta \left( {x + u} \right) * \phi _3 \left( x \right)} \right]dud\tau dy} } }
\]
\[
 - \frac{1}{{\left( {a^2 b - b^3 } \right)}}\frac{1}{2}\int\limits_{ - bt}^{bt} {\int\limits_0^y {\int\limits_0^\tau  {\mathcal{F}\left[ {\delta \left( {x + u} \right) * \phi _3 \left( x \right)} \right]dud\tau dy} } }.
\]
Consequently, we get the inverse formula of $\widehat{u}$ given by
\[
u\left( {x,t} \right) =  - \frac{{b^2 }}{{\left( {a^2  - b^2 } \right)}}\frac{{\left( {\delta \left( {x + at} \right) * \phi _0 \left( x \right)} \right) + \left( {\delta \left( {x - at} \right) * \phi _0 \left( x \right)} \right)}}{2}
\]
\[
 + \frac{{a^2 }}{{\left( {a^2  - b^2 } \right)}}\frac{{\left( {\delta \left( {x + bt} \right) * \phi _0 \left( x \right)} \right) + \left( {\delta \left( {x - bt} \right) * \phi _0 \left( x \right)} \right)}}{2}
\]
\[
 - \,\frac{{b^2 }}{{\left( {a^3  - ab^2 } \right)}}\frac{1}{2}\int\limits_{ - at}^{at} {\left( {\delta \left( {x + \theta } \right) * \phi _1 \left( x \right)} \right)d\theta }  + \frac{{a^2 }}{{\left( {a^2 b - b^3 } \right)}}\frac{1}{2}\int\limits_{ - bt}^{bt} {\left( {\delta \left( {x + \theta } \right) * \phi _1 \left( x \right)} \right)d\theta }
\]
\[
 - \frac{1}{{\left( {a^2  - b^2 } \right)}}\left( { - \frac{1}{2}\int\limits_0^{at} {\int\limits_0^y {\left( {\delta \left( {x + u} \right) * \phi _2 \left( x \right)} \right)du} dy}  - \frac{1}{2}\int\limits_0^{at} {\int\limits_0^y {\left( {\delta \left( {x - u} \right) * \phi _2 \left( x \right)} \right)du} dy} } \right)
\]
\[
 + \frac{1}{{\left( {a^2  - b^2 } \right)}}\left( { - \frac{1}{2}\int\limits_0^{bt} {\int\limits_0^y {\left( {\delta \left( {x + u} \right) * \phi _2 \left( x \right)} \right)du} dy}  - \frac{1}{2}\int\limits_0^{bt} {\int\limits_0^y {\left( {\delta \left( {x - u} \right) * \phi _2 \left( x \right)} \right)du} dy} } \right)
\]
\[
 + \,\frac{1}{{\left( {a^3  - ab^2 } \right)}}\frac{1}{2}\int\limits_{ - at}^{at} {\int\limits_0^y {\int\limits_0^\tau  {\left( {\delta \left( {x + u} \right) * \phi _3 \left( x \right)} \right)dud\tau dy} } }
\]
\[
 - \frac{1}{{\left( {a^2 b - b^3 } \right)}}\frac{1}{2}\int\limits_{ - bt}^{bt} {\int\limits_0^y {\int\limits_0^\tau  {\left( {\delta \left( {x + u} \right) * \phi _3 \left( x \right)} \right)dud\tau dy} } }.
\]
The last formula is equivalent to  the one given at (\ref{eq0}), so the theorem is proved.
\end{proof}
\ \\
For the generalized case, we will use the following result:
\begin{Lemma} For an odd number $n\ge 3$, $\displaystyle m=\frac{n-3}{2}$ and $0\le k\le m$ then
\[\displaystyle \int_{-R}^R e^{is|\xi|}(R^2-s^2)^{m-k} ds=\frac{1}{2^{k}k!}\left(\frac{1}{R}\frac{\partial}{\partial R}\right)^{k} \left( \frac{1}{\omega_{n-1}R} \int_{\partial B(0,R)}e^{-ix.\xi}d\sigma(x)\right).\]
\end{Lemma}

For the proof, we refer the reader to Torchinsky's paper in \cite{2009ARX.AP}. Note that, in the case $k=m$, it follows that
\begin{equation}\label{eq10}
\frac{\sin (a|\xi|t) }{|\xi|}=\frac{1}{2}\int_{-at}^{at}e^{is|\xi|}ds
=\frac{1}{2^{m+1}m!}\left(\frac{1}{a^2 t}\frac{\partial}{\partial t}\right)^{m}\left( \frac{1}{\omega_{n-1}at} \int_{\partial B(0,at)}e^{-ix.\xi}d\sigma(x)\right).\end{equation}
Differentiating with respect to $t$, then
\begin{equation}\label{eq11}
\cos (a|\xi|t)=\frac{1}{2^{m+1}m!a}\frac{\partial}{\partial t}\left(\frac{1}{a^2 t}\frac{\partial}{\partial t}\right)^{m}\left( \frac{1}{\omega_{n-1}at} \int_{\partial B(0,at)}e^{-ix.\xi}d\sigma(x)\right).
\end{equation}
On the other hand, we have
\[\frac{\cos(b|\xi|t)}{|\xi|^2}-\frac{\cos(a|\xi|t)}{|\xi|^2}=\int_{bt}^{at}
\frac{\sin(s|\xi|)}{|\xi|}ds= \frac{1}{2}\int_{bt}^{at} \int_{-s}^se^{i\tau|\xi|}d\tau ds\]
\begin{equation}\label{eq12}
=\int_{bt}^{at}\frac{1}{2^{m+1}m!}\left(\frac{1}{s}\frac{\partial}{\partial s}\right)^{m}\left( \frac{1}{\omega_{n-1}s} \int_{\partial B(0,s)}e^{-ix.\xi}d\sigma(x)\right)ds.
\end{equation}
Integrating the above formula with respect to $t$, it implies that
\[
\frac{1}{b}\frac{\sin(b|\xi|t)}{|\xi|^3}-\frac{1}{a}\frac{\sin(a|\xi|t)}{|\xi|^3}
=\frac{1}{2}\int_{0}^t\int_{b\nu}^{a\nu}\int_{-s}^s e^{i\tau|\xi|}d\tau ds d\nu\]
\begin{equation}\label{eq13}
=\int_0^t \int_{b\nu}^{a\nu}\frac{1}{2^{m+1}m!}\left(\frac{1}{s}\frac{\partial}{\partial s}\right)^{m}\left( \frac{1}{\omega_{n-1}s} \int_{\partial B(0,s)}e^{-ix.\xi}d\sigma(x)\right)ds d\nu. \end{equation}
Moreover, for each function $\theta\in \mathcal{S}(\mathbb{R}^n)$, we also have
\[\int_{\mathbb{R}^n}\frac{\sin (a|\xi|t) }{|\xi|}\theta(\xi)d\xi
=\frac{1}{2^{m+1}m!}\left(\frac{1}{a^2 t}\frac{\partial}{\partial t}\right)^{m}\left( \frac{1}{\omega_{n-1}at} \int_{\partial B(0,at)}\int_{\mathbb{R}^n} e^{-ix.\xi}\theta(\xi)d\xi d\sigma(x)\right)\]
\[=\frac{1}{2^{m+1}m!}\left(\frac{1}{a^2 t}\frac{\partial}{\partial t}\right)^{m}\left( \frac{1}{\omega_{n-1}at} \int_{\partial B(0,at)}\widehat{\theta}(x) d\sigma(x)\right).\]
Therefore, we conclude that
\[\frac{1}{2^{m+1}m!}\left(\frac{1}{a^2 t}\frac{\partial}{\partial t}\right)^{m}\left( \frac{1}{\omega_{n-1}at} \int_{\partial B(0,at)} d\sigma(x)\right)^{\widehat{\  \ \ }}(\xi)=\frac{\sin (a|\xi|t) }{|\xi|}.
\]
By the Fourier inversion and convolution formulas, we obtain the identity
\[\begin{array}{llll} \displaystyle
\frac{1}{(2\pi)^n} \int_{\mathbb{R}^n} \widehat{\phi_1}(\xi) \frac{\sin (a|\xi|t) }{|\xi|}
e^{i\xi.x}d\xi \\
\displaystyle
= \frac{1}{2^{m+1}m!}\left(\frac{1}{a^2 t}\frac{\partial}{\partial t}\right)^{m}\left( \frac{1}{\omega_{n-1}at} \int_{\partial B(x,at)}\phi_1(y) d\sigma(y) \right)
\\ \displaystyle
=\frac{\omega_n}{2^{m+1}m!\omega_{n-1} }\left(\frac{1}{a^2 t}\frac{\partial}{\partial t}\right)^{m}
\left((at)^{n-2} M_{at}(\phi_1)(x)\right)
\\ \displaystyle
=\frac{1}{(n-2)!!}\left(\frac{1}{a^2 t}\frac{\partial}{\partial t}\right)^{m}
\left( (at)^{n-2} M_{at}(\phi_1)(x)\right).
\end{array}
\]
Applying the same way for the expressions (\ref{eq11})-(\ref{eq13}) and substituting the obtained identities into the formula (\ref{fou}), we have found an exact solution to the $n$-dimensional biwave equation, where $n\ge 3$ is an odd number:
\begin{Theorem}\label{th3} The Cauchy initial value problem for the homogeneous $n$-dimensional biwave equation, where $n\ge 3$ is an odd number, has the following solution
\[u(x,t)=\frac{1}{(n-2)!!(a^2-b^2)}\left[
\frac{a^2}{b} \frac{\partial}{\partial t}\left( \frac{1}{b^2 t} \frac{\partial }{\partial t} \right)^{\frac{n-3}{2}} \left( (bt)^{n-2} M_{bt}(\phi_0)(x) \right)- \right.\]
\[ - \frac{b^2}{a} \frac{\partial}{\partial t}\left( \frac{1}{a^2 t} \frac{\partial }{\partial t} \right)^{\frac{n-3}{2}} \left( (at)^{n-2} M_{at}(\phi_0)(x) \right)+
\]
\[
+\frac{a^2}{b} \left( \frac{1}{b^2 t} \frac{\partial }{\partial t} \right)^{\frac{n-3}{2}} \left( (bt)^{n-2} M_{bt}(\phi_1)(x) \right)-\frac{b^2}{a} \left( \frac{1}{a^2 t} \frac{\partial }{\partial t} \right)^{\frac{n-3}{2}} \left( (at)^{n-2} M_{at}(\phi_1)(x) \right)+
\]
\[\left.
+\int_{bt}^{at} \left( \frac{1}{s} \frac{\partial }{\partial s} \right)^{\frac{n-3}{2}} \left( (s)^{n-2} M_{s}(\phi_2)(x) \right)ds\right.\]
\[\left.+
\int_0^t \int_{b\nu}^{a\nu} \left( \frac{1}{s} \frac{\partial }{\partial s} \right)^{\frac{n-3}{2}} \left( (s)^{n-2} M_{s}(\phi_3)(x) \right)dsd\nu
 \right].
\]
\end{Theorem}

Now we consider the case when $n$ is an even number. The Hadamard's method of descent (see e.g. \cite{Hadamard53}) is useful to connect with the case in the odd dimensional space $\mathbb{R}^{n+1}$. For fixed $T>0$, we choose a Schwartz function $\eta\in \mathcal{S}(\mathbb{R})$, such that $\eta(x_{n+1})=1$ for all $|x_{n+1}|\le nT$. Let us denote $$\overline{\phi_i}(x_1,x_2,...,x_n,x_{n+1})={\phi_i}(x_1,x_2,...,x_n)\eta(x_{n+1}), \ \ i=0,1,2,3.$$
It is easy to see that $\overline{\phi_i}\in \mathcal{S}(\mathbb{R}^{n+1})$. For $|x_{n+1}|\le T,t\le T$, the solution $\overline{u}(x_1,x_2,...,x_{n+1},t)$ to the Cauchy problem for $n+1$-dimensional biwave equation with initial valued functions $\overline{\phi_i},\ i=0,1,2,3$ does not depend on $x_{n+1}$. In particular, $$u(x_1,x_2,...,x_n,t)=\overline{u}(x_1,x_2,...,x_n,0,t)$$ is the solution to the $n$-dimensional biwave equation for all $|t|\le T$. Since $T$ is arbitrary, so $u$ is the solution to the Cauchy problem in even dimensional space $\mathbb{R}^n$.

\begin{Lemma}
Given a function $f:\mathbb{R}^{n+1}\to \mathbb{R}$, which does not depend on the last variable, i.e. $f(x_1,x_2,...,x_{n+1})=g(x_1,x_2,...,x_n)$, then
\[\mathcal{M}_t(f)(x,0)=\frac{2}{\omega_{n+1}}\int_{B_n(0,1)}\frac{g(x+tz)}{\sqrt{1-|z|^2}}dz.\]
\end{Lemma}

\begin{proof}
Observe that, for $\widetilde{x}=(x,0)$ and $\widetilde{y}=(y,y_{n+1})$, we have
\[\mathcal{M}_t(f)(\widetilde{ x})=\frac{1}{\omega_{n+1}}\int_{\partial B_{n+1}(0,1)}f(\widetilde{ x}+t\widetilde{y})d\sigma(\widetilde{y}).\]
We use the spherical coordinates given by
\[
\left\{ \begin{array}{l}
 y_1  = \sin \varphi _1 \sin \varphi _2 ...\sin \varphi _{n - 2} \sin \varphi _{n - 1} \sin \varphi _n,  \\
 y_2  = \sin \varphi _1 \sin \varphi _2 ...\sin \varphi _{n - 2} \sin \varphi _{n - 1} \cos \varphi _n , \\
 y_3  = \sin \varphi _1 \sin \varphi _2 ...\sin \varphi _{n - 2} \cos \varphi _{n - 1} \cos \varphi _n  ,\\
 ... \\
 y_n  = \sin \varphi _1 \cos \varphi _2 , \\
 y_{n + 1}  = \cos \varphi _1 , \\
 \end{array} \right.
\]
where $0\le\varphi_k\le \pi, k=1,2,...,n-1$ and $0\le\varphi_{n}\le 2\pi$. The Jacobian of this transformation is calculated as $$J=\sin^{n-1}\varphi_1\sin^{n-2} \varphi_2...\sin \varphi_{n-1}.$$
Therefore
$$\displaystyle\mathcal{M}_t(f)(x,0)=\frac{1}{\omega_{n+1}}\int_0^{\pi}...\int_0^{\pi}\int_0^{2\pi}g(x+ty)J d\varphi_1d\varphi_2...d\varphi_n.$$
Let us give $r=\sin\varphi_1$, and
\[
\left\{ \begin{array}{l}
 z_1  = r\sin \varphi _2 ...\sin \varphi _{n - 1} \sin \varphi _n,  \\
 z_2  = r\sin \varphi _2 ...\sin \varphi _{n - 1} \cos \varphi _n , \\
 ... \\
 z_n  = r\cos \varphi _2 .\\
 \end{array} \right.
\]
The Jacobian of above transformation is calculated by the formula
$$J'=\displaystyle\frac{1}{r^{n-1} \sin^{n-2}\varphi_2\sin^{n-3}\varphi_3...\sin \varphi_{n-1}}.$$
Finally, we obtain that
\[\begin{array}{llll} \displaystyle
\mathcal{M}_t(f)(x,0)=\frac{2}{\omega_{n+1}}\int_0^{\pi}...\int_0^{\pi}\int_0^{2\pi}\int_0^{1}g(x+tz)\frac{1}{\cos\varphi_1}J drd\varphi_2...d\varphi_n\\ \displaystyle =\frac{2}{\omega_{n+1}}\int_0^{\pi}...\int_0^{\pi}\int_0^{2\pi}\int_0^{1}g(x+tz)\frac{1}{\sqrt{1-|z|^2}}J drd\varphi_2...d\varphi_n\\
\displaystyle=\frac{2}{\omega_{n+1}}\int_0^{\pi}...\int_0^{\pi}\int_0^{2\pi}\int_0^{1}g(x+tz)\frac{1}{\sqrt{1-|z|^2}}JJ' dz_1dz_2...dz_n\\
=\displaystyle\frac{2}{\omega_{n+1}}\int_{B_n(0,1)}\frac{g(x+tz)}{\sqrt{1-|z|^2}}dz.\end{array}\]
So the lemma is proved.
\end{proof}

We use the notation $\displaystyle \widetilde{\mathcal{M}}_t(f)(x)=\frac{2}{\omega_{n+1}}\int_{B_n(0,1)}\frac{f(x+tz)}{\sqrt{1-|z|^2}}dz,$ which is called as modified spherical mean of $f$ (see e.g. \cite{Sabelfeld97, Stein03}). Applying the result of Lemma 2.4, we obtain the formula of the solution to the biwave equation in the even dimensional space $\mathbb{R}^n$:
\begin{Theorem}
The Cauchy initial value problem for the homogeneous $n$-dimensional biwave equation, where $n\ge 2$ is an even number, has the following solution
\[u(x,t)=\frac{1}{(n-1)!!(a^2-b^2)}\left[
\frac{a^2}{b} \frac{\partial}{\partial t}\left( \frac{1}{b^2 t} \frac{\partial }{\partial t} \right)^{\frac{n-2}{2}} \left( (bt)^{n-1} \widetilde{\mathcal{M}}_{bt}(\phi_0)(x) \right) \right.\]
\[ - \frac{b^2}{a} \frac{\partial}{\partial t}\left( \frac{1}{a^2 t} \frac{\partial }{\partial t} \right)^{\frac{n-2}{2}} \left( (at)^{n-1} \widetilde{\mathcal{M}}_{at}(\phi_0)(x) \right)
\]
\[+
\frac{a^2}{b} \left( \frac{1}{b^2 t} \frac{\partial }{\partial t} \right)^{\frac{n-2}{2}} \left( (bt)^{n-1} \widetilde{\mathcal{M}}_{bt}(\phi_1)(x) \right)-\frac{b^2}{a} \left( \frac{1}{a^2 t} \frac{\partial }{\partial t} \right)^{\frac{n-2}{2}} \left( (at)^{n-1} \widetilde{\mathcal{M}}_{at}(\phi_1)(x) \right)
\]
\[\left.
+\int_{bt}^{at} \left( \frac{1}{s} \frac{\partial }{\partial s} \right)^{\frac{n-2}{2}} \left( (s)^{n-1} \widetilde{\mathcal{M}}_{s}(\phi_2)(x) \right)ds\right.\]
\[\left.-
\int_0^t \int_{b\nu}^{a\nu} \left( \frac{1}{s} \frac{\partial }{\partial s} \right)^{\frac{n-2}{2}} \left( (s)^{n-1} \widetilde{\mathcal{M}}_{s}(\phi_3)(x) \right)dsd\nu
 \right].
\]
\end{Theorem}
\ \\

By a similar idea with the Duhamel principle for wave equations, the solution of the Cauchy problem for the nonhomogeneous biwave equation will be given at the next theorem
\begin{Theorem}\label{th2} The solution of the equation (\ref{equ1})-(\ref{equ2}) takes the form $u = \widetilde u + v$, where $\widetilde u$  is the solution of the equation (\ref{equ3})-(\ref{equ4}) and \[\displaystyle v\left( {x,t} \right) = \int\limits_0^t {\omega \left( {x,t,\tau } \right)} d\tau, \]
 where $\omega \left( {x,t,\tau } \right)$ is the solution of the homogeneous biwave equation
\[
\left( \frac{{\partial ^2 }}{{\partial t^2 }} - a^2{\rm\Delta} \right)\left( {\frac{{\partial ^2 \omega }}{{\partial t^2 }} - b^2{\rm\Delta}\omega } \right) = 0,\,\,\,\,t > \tau,
\]
with the initial conditions
\[
\left. \omega  \right|_{t = \tau }  = 0,\left. {\frac{{\partial \omega }}{{\partial t}}} \right|_{t = \tau }  = 0,\left. {\frac{{\partial ^2 \omega }}{{\partial t^2 }}} \right|_{t = \tau }  = 0,\left. {\frac{{\partial ^3 \omega }}{{\partial t^3 }}} \right|_{t = \tau }  = f\left( { x,\tau} \right).
\]

\end{Theorem}
\begin{proof}
We start with the observation that
\[
\frac{{\partial^4 v}}{{\partial t^4 }} - \left( {a^2  + b^2 } \right)\frac{{\partial ^2}}{{\partial t^2}}{\rm \Delta} v + a^2 b^2 {\rm \Delta}^2 v
\]
\[
\begin{array}{l}
 \displaystyle = f\left( {x,t} \right) + \int\limits_0^t {\left( {\frac{{\partial ^4 \omega }}{{\partial t^4 }} - \left( {a^2  + b^2 } \right)\frac{{\partial ^2  }}{{\partial t^2  }}{\rm \Delta} \omega + a^2 b^2{\rm \Delta}^2 \omega} \right)} d\tau  \\
 \displaystyle = f\left( {x,t} \right). \\
 \end{array}.
\]
Then, the above identity follows that
\[
\frac{{\partial ^4 u}}{{\partial t^4 }} - \left( {a^2  + b^2 } \right)\frac{{\partial ^2  }}{{\partial t^2  }}{\rm \Delta} u + a^2 b^2 {\rm \Delta}^2u = f\left( {x,t} \right).
\]
Moreover,
\[
\left. u \right|_{t = 0}  = \left. {\widetilde u} \right|_{t = 0}  + \left. v \right|_{t = 0}  = \phi _0 \left( x \right) + 0 = \phi _0 \left( x \right),
\]
\[
\left. {\frac{{\partial u}}{{\partial t}}} \right|_{t = 0}  = \left. {\frac{{\partial \widetilde u}}{{\partial t}}} \right|_{t = 0}  + \left. {\frac{{\partial v}}{{\partial t}}} \right|_{t = 0}  = \phi _1 \left( x \right) + 0 = \phi _1 \left( x \right),
\]
\[
\left. {\frac{{\partial ^2 u}}{{\partial t^2 }}} \right|_{t = 0}  = \left. {\frac{{\partial ^2 \widetilde u}}{{\partial t^2 }}} \right|_{t = 0}  + \left. {\frac{{\partial ^2 v}}{{\partial t^2 }}} \right|_{t = 0}  = \phi _2 \left( x \right) + 0 = \phi _2 \left( x \right),
\]
\[
\left. {\frac{{\partial ^3 u}}{{\partial t^3 }}} \right|_{t = 0}  = \left. {\frac{{\partial ^3 \widetilde u}}{{\partial t^3 }}} \right|_{t = 0}  + \left. {\frac{{\partial ^3 v}}{{\partial t^3 }}} \right|_{t = 0}  = \phi _3 \left( x \right) + 0 = \phi _3 \left( x \right).
\]
So the theorem is proved.
\end{proof}

\section{Example}\label{s:3}

Let us give an example demonstrating Theorem \ref{th1}.\\
\ \\
Consider the equation
\begin{equation}\label{equ6}
\left( {\frac{{\partial ^2 }}{{\partial t^2 }} - \frac{{\partial ^2 }}{{\partial x^2 }}} \right)\left( {\frac{{\partial ^2 u}}{{\partial t^2 }} - \frac{1}{4}\frac{{\partial ^2 u}}{{\partial x^2 }}} \right) = 0
\end{equation}
with the initial conditions
\begin{equation}\label{equ7}
\left. u \right|_{t = 0}  = 0,\left. {\frac{{\partial u}}{{\partial t}}} \right|_{t = 0}  = \sin x,\left. {\frac{{\partial ^2 u}}{{\partial t^2 }}} \right|_{t = 0}  = \cos x,\left. {\frac{{\partial ^3 u}}{{\partial t^3 }}} \right|_{t = 0}  = 0.
\end{equation}
The solution of the equation (\ref{equ6})-(\ref{equ7}) is given by the formula
\[
u\left( {x,t} \right) = \frac{1}{3}\left( {4\cos \left( {\frac{t}{2}} \right)\cos x - 4\cos t\cos x - \left( { - 8\sin \left( {\frac{t}{2}} \right) + \sin t} \right)\sin x} \right).
\]



\bibliographystyle{amsalpha}
\bibliography{x}

\end{document}